\documentclass[reqno,12pt]{amsart}
\usepackage{amsmath}
\usepackage{amssymb}
\usepackage{cite}
 
\usepackage{amsmath}
\usepackage{txfonts}
\usepackage{stmaryrd}
\usepackage{amssymb}
\usepackage{amsfonts}
\usepackage{times}
\usepackage{mathrsfs}
\usepackage{amsthm}
\usepackage{enumerate}
\usepackage{framed}

\newtheorem{thm}{Theorem}
\newtheorem{lemma}[thm]{Lemma}

\def\D{\mathbb{D}}

\def\T{\partial \D}
\def\C{\mathbb{C}}

\def\f{\frac}

\def\a{\alpha}

\def\msk{\medskip}

\def\bege{\begin{equation}} \def\ende{\end{equation}}
  \def\a{\alpha}

\def\begr{\begin{eqnarray}} \def\endr{\end{eqnarray}}

\def\bege{\begin{equation}} \def\ende{\end{equation}}
\def\begr{\begin{eqnarray}} \def\endr{\end{eqnarray}}
\def\bnum{\begin{enumerate}} \def\enum{\end{enumerate}}

\begin{document}

\title[Difference of composition-differentiation operators]{Difference of composition-differentiation operators from Hardy spaces  to weighted Bergman spaces via harmonic analysis}
\author{Yecheng Shi and Songxiao Li$^*$  }

\address{Yecheng Shi\\    School of Mathematics and Statistics, Lingnan Normal University,
 Zhanjiang 524048, Guangdong, P. R. China}\email{ 09ycshi@sina.cn}

 \address{Songxiao Li\\ Institute of Fundamental and Frontier Sciences, University of Electronic Science and Technology of China,
610054, Chengdu, Sichuan, P. R. China.   } \email{jyulsx@163.com}

\subjclass[2010]{30H10, 30H20, 47B38}
\begin{abstract}  In this paper, the boundedness and compactness of the difference of composition-differentiation operators $D_\varphi-D_\psi$ acting from Hardy spaces $H^p$ to weighted Bergman spaces $A^q_\alpha$ are completely characterize for all $0<p,q<\infty$.
\thanks{*Corresponding author.}
\vskip 3mm \noindent{\it Keywords}: Hardy space; weighted Bergman space; composition operator; difference;  norm.
\end{abstract}

\maketitle

\section{Introduction}

Let $\D$ and $\partial\D$ denote the open unit disk and the unit circle of the complex plane $\C$, respectively. In the sequel, $dm=\frac{d\theta}{2\pi}$  will be the normalized Lebesgue measure on $\partial\D$. The Lebesgue space $L^p(\partial\D,m)$ will also be denoted by $L^p(\partial\D)$.  We denote by $H(\D)$ the class of all functions analytic in $\D$.

For $0<p<\infty,$ let $H^p$ be the classical Hardy space of all  $f \in H(\D)$  such that
$$\|f\|_{p}^p=\sup_{0<r<1}\frac{1}{2\pi}\int_{0}^{2\pi}|f(re^{i\theta})|^pd\theta<\infty.$$

 Given $\alpha>-1$ and $0<p<\infty$, a function $f\in H(\D)$ belongs to the weighted Bergman spaces $A^p_\alpha$, if
$$\|f\|_{A^p_\alpha}^p=\int_{\D}|f(z)|^pdA_\alpha(z)<\infty.$$
Here $dA$ is the Lebesgue measure on $\D$, normalized so that $A(\D)=1$. The measure $dA_\alpha$ is given by $dA_\alpha=(\alpha+1)(1-|z|^2)^\alpha dA(z)$.
  We refer the reader to \cite{Duren,DS,Ga, Zhu,zhu} for more information of Hardy spaces and   weighted Bergman spaces.

Let $\varphi$ be an analytic self-map of $\D$. The composition operator $C_\varphi$ on $H(\D)$ is defined by
$$C_\varphi f=f\circ\varphi, ~~~~f\in H(\D).$$
Let $D_\varphi$  denote the operator given by
$$D_\varphi f=f^\prime\circ\varphi, ~~~~f\in H(\D).$$
It has been of growing interest to study the difference of composition operators for the last three decades. The question of when the difference of two composition operators is compact on Hardy spaces $H^p$ was posed by Shapiro and Sundberg \cite{SS} in 1990.
The difference of composition operators has been studied  by many authors
on several function spaces. See \cite{G,KW,LS,LRW,SL,SLD, SQL} for example.
In 2004, Nieminen and Saksman \cite{NS} showed that the compactness of $C_\varphi-C_\psi$ on   $H^p$, with $1\leq p<\infty$, is independent of $p$.  Very recently, in \cite{CCKY}, Choe et al. completely characterized compact operators $C_\varphi-C_\psi$ on $H^p$ by using Carleson measures of Bergman spaces.
  Moorhouse \cite{Mo} characterized the compactness of $C_\varphi-C_\psi$ on the standard weighted
Bergman space $A^2_\alpha$. Saukko \cite{S1,S2} generalized Moorhouse's results by
characterizing the  boundedness and compactness from $A^p_\alpha$ to $A^q_\beta$. But Saukko's proof
contains some serious problems. Very recently, in \cite{LSS}, Lindstr\"om, Saukko and the first author of this paper have corrected Saukko's proof.

In the context of analytic functions on $\D$, it also seems reasonable to consider the difference of composition-differentiation operators.
The aim of this paper is to characterize the boundedness and compactness of the difference of composition-differentiation operators $D_\varphi-D_\psi:H^p\to A^q_\alpha$. The Littlewood-Paley formula of standard Bergman spaces
$$\|f\|_{A^p_\alpha}\asymp\sum_{j=0}^{n-1}|f^{(j)}(0)|+\|f^{(n)}\|_{A^p_{\alpha+np}}$$
implies that the analogous problem for $D_\varphi-D_\psi:A^p_\alpha\to A^q_\beta$ is equivalent to characterize the boundedness and compactness of the difference of composition operators $C_\varphi-C_\psi:A^p_{\alpha+p}\to A^q_\beta$, which  have been solved in \cite{LSS,Mo,S1,S2}. However, this method does not work for our setting, because such a Littlewood-Paley formula does not exist for Hardy spaces when $p\neq2$.

Let us now recall some definitions that will enable us to state the solution to the primary question of this paper.
For $a\in\D$, let $\sigma_{a}(z) =\frac{a-z}{1-\bar{a}z}$ be the disc automorphism that exchanges 0 for $a$.
For two points $z$, $w\in\D$, the pseudo-hyperbolic distance is given by
$\rho(z,w)=|\sigma_{w}(z)|=\big|\frac{z-w}{1-\bar{w}z}\big|.$ We denote $\Delta(a,r) = \{z \in\D: |\sigma_a(z)| < r\}$ by the pseudohyperbolic disk centered at $a$ with radius $r$.
Given two holomorphic self-maps $\varphi, \psi$ of $\D$, we put
$$\rho(z)=\rho(\varphi(z),\psi(z))$$
for short.  For $0<p,q<\infty$ and $-1<\alpha<\infty$,   the joint pull-back measure
$ \mu$ is defined by
\begr\mu(E):=\mu_{q,\alpha,\varphi,\psi}(E)=\int_{\varphi^{-1}(E)}\rho(z)^qdA_\alpha(z)+\int_{\psi^{-1}(E)}\rho(z)^qdA_\alpha(z)\label{m}\endr
for any Borel set $E\subset\D$.  Here and henceforth   $\mu$ denotes the join pull-back measure defined as above.
$\mu$ is actually the sum of two pull-back measures
$\rho^qdA_\alpha\circ\varphi^{-1}$ and $\rho^qdA_\alpha\circ\psi^{-1}$. By a standard argument one can verify that
\begin{equation} \label{gs}
\int_{\D}gd\mu=\int_{\D}(g\circ\varphi+g\circ\psi)\rho^q dA_\alpha
\end{equation}
for any positive Borel function $g$ on $\D$.

We state our main result in this paper as follows.\msk

\begin{thm}\label{main}
 Suppose $\varphi$ and $\psi$ are analytic self-maps of $\D$.  Let $0<r<1$, $0<p,q< \infty$, and let $\mu$ denote the joint pull-back measure  induced by $\varphi$ and $\psi$ defined as (\ref{m}).\\
(i) If either $0<p<q< \infty$ or $2\leq p=q$, then \msk

(ia) $D_\varphi-D_\psi:H^p\to A^q_\alpha$ is bounded if and only if
\begr\sup_{z\in\D}\frac{\mu(\Delta(z,r))}{(1-|z|^2)^{\frac{q}{p}+q}}<\infty.\nonumber
\endr\msk

(ib) $D_\varphi-D_\psi:H^p\to A^q_\alpha$ is compact if and only if
\begr\lim_{|z|\to1}\frac{\mu(\Delta(z,r))}{(1-|z|^2)^{\frac{q}{p}+q}}=0.\nonumber\endr
(ii)~If $q<\min\{2,p\}$, then the following statements are equivalent:

(iia) $D_\varphi-D_\psi:H^p\to A^q_\alpha$ is bounded;

(iib) $D_\varphi-D_\psi:H^p\to A^q_\alpha$ is compact;

(iic) The function \begr\zeta\mapsto\left(\int_{\Gamma(\zeta)}\left(\frac{\mu(\Delta(z,r))}{(1-|z|^2)^{1+q}}\right)^{\frac{2}{2-q}}\frac{dA(z)}{(1-|z|)^2}\right)^{\frac{2-q}{2}}\nonumber
\endr
\par belongs to $L^{\frac{p}{p-q}}(\partial\D,m)$.\\
(iii)~~~If~$0<q=p<2$, then

(iiia) $D_\varphi-D_\psi:H^p\to A^q_\alpha$ is bounded if and only if
\begr \left(\frac{\mu(\Delta(z,r))}{(1-|z|^2)^{1+p}}\right)^{\frac{2}{2-p}}\frac{dA(z)}{1-|z|} ~~ \nonumber \endr
  is a Carleson measure.

(iiib) $D_\varphi-D_\psi:H^p\to A^q_\alpha$ is compact if and only if
\begr \left(\frac{\mu(\Delta(z,r))}{(1-|z|^2)^{1+p}}\right)^{\frac{2}{2-p}}\frac{dA(z)}{1-|z|} ~~ \nonumber \endr
 is a  vanishing Carleson measure.\\
(iv)~If $2\leq q<p<\infty$, then

(iva) $D_\varphi-D_\psi:H^p\to A^q_\alpha$ is bounded if and only if the function
 \begr
\zeta\mapsto\sup_{z\in\Gamma(\zeta)}\frac{\mu(\Delta(z,r))}{(1-|z|^2)^{1+q}}\nonumber
\endr
  belongs to $L^{\frac{p}{p-q}}(\partial\D,m)$.

(ivb) $D_\varphi-D_\psi:H^p\to A^q_\alpha$ is compact if and only if the function
 \begr
\zeta\mapsto\sup_{z\in\Gamma(\zeta)\cap\{|z|\geq r\}}\frac{\mu(\Delta(z,r))}{(1-|z|^2)^{1+q}}\nonumber
\endr
\par converges to zero in $L^{\frac{p}{p-q}}(\partial\D,m)$ as $r\to1^-$.
\end{thm}

The present paper is organized as follows. In Section 2, we collect some notations and standard facts that will be used in Section 3. In  Section 4, we give the proof of Theorem 1. For two quantities $A$ and $B$, we use the abbreviation
$A\lesssim B$ whenever there is a positive constant $c$ (independent of the associated variables) such that $A\leq cB$.
We write $A\asymp B$, if $A\lesssim B\lesssim A$. Given $p\in[1,\infty]$, we will denote by $p' = p/(p-1)$ its H\"older conjugate. In this paper we
agree that $1'=\infty$ and $\infty '=1$.\msk

\section{prerequisites}

In this section, we collect some well-known results that will be used throughout the paper.\msk

\subsection{Carleson measures}

We recall that a positive Borel measure $\mu$ on $\D$ is a Carleson measure if there is a positive constant $C>0$ such that
$$\int_{\D}|f(z)|^pd\mu(z)\leq C\|f\|_{H^p}^p$$
for any $f\in H^p$.
It is well known that, $\mu$ on $\D$ is a Carleson measure if and only if there exists a constant $C>0$ such that
$$\mu(Q_r(\zeta))\leq Cr$$
for all $\zeta\in\partial\D$ and $r>0$, where $Q_r(\zeta)=\{z\in\D:|1-\overline{\zeta}z|<r\}$ is called a Carleson region at $\zeta$. Obviously every Carleson measure is finite.  We say that a positive finite Borel measure $\mu$ is a vanishing Carleson measure if for any sequence $\{f_k\}$ in $H^p$ with $\|f\|_{H^p}\leq1$ and $f_k\to0$ uniformly on any compact subset of $\D$,
$$\lim_{k\to\infty}\int_{\D}|f_k(z)|^pd\mu(z)=0.$$
It is well known that, a positive finite Borel measure $\mu$ is a vanishing Carleson measure if and only if
\begr
\lim_{r\to0}\frac{\mu(Q_r(\zeta))}{r}=0\nonumber
\endr
uniformly for $\zeta\in\D$. We refer the reader to \cite{Duren,Ga, Zhu,zhu} for more information of  Carleson measure and vanishing Carleson measure.

\subsection{Separated sequences and lattices}

Recall that a sequence $\{z_k\}\subset\D$ is called a $\delta$-separated sequence if there exists $\delta>0$ such that $\rho(z_i,z_j)\geq\delta$ for all $i$ and $j$ with $i\neq j$. A sequence $\{z_k\}\subset\D$ is called a $\delta$-lattice in the pseudo-hyperbolic metric if it is $\delta$-separated and $\D=\bigcup\limits_{k=0}^{\infty}\Delta(z_k,\delta)$. The following result is similar with \cite[Lemma 4]{Lu4}.

\begin{lemma}\label{e}  Assume $0<r<1$. Then there exists a  $\delta$-lattice $\{z_k\}$ such that $\inf_k|z_k|\geq r$ for some $0<\delta<1$.
\end{lemma}

\begin{proof} Denote $\delta=\frac{2r}{1+r^2}$. Let $z_1$ be any point with $\rho(z_1,0)=r$, and let $z_2$ be any point with $\rho(z_1,z_2)=\delta.$ Given $z_1,z_2,...,z_{k-1}$, select $z_k\notin\bigcup_{j=1}^{k-1}\Delta(z_j,\delta)$ such that $\rho(z_1,z_k)$ is minimized. This defines the sequence $\{z_k\}$ inductively. Clearly $\rho(z_k,z_n)\geq \delta$ for any $k\neq n$ and $\inf_k|z_k|\geq r$.
In fact, $$|z_k|=\rho(0,z_k)\geq\frac{\rho(z_1,z_k)-\rho(z_1,0)}{1-\rho(z_1,z_k)\rho(z_1,0)}\geq\frac{\delta-r}{1-\delta r}= r,$$
for any $k\neq 1$ and $|z_1|=r$.

We now show that $\D=\bigcup_{k=0}^{\infty}\Delta(z_k,\delta).$ We will do this by using a proof by contradiction.
Assume that $z \notin \cup_{k=1}^{\infty}\Delta(z_k,\delta), z\in\D$. Denote $s=\rho(z_1,z)$ and $\lambda=\frac{2s}{1+s^2}$. By the choice of $z_k$ for all $k$ we have $\rho(z_1,z_k)\leq s<\lambda$.
On the other hand, by \cite[Chapter 2,Lemma 15]{DS} or \cite[Lemma 3]{Lu4}, there exists a constant $K$ such that
$z_1$ belongs to at most $K$ discs $\triangle(z_k,\lambda)$, a contradiction. Therefore, $\D=\bigcup_{k=0}^{\infty}\Delta(z_k,\delta)$. The proof is complete.
\end{proof}

\noindent{\bf Remark.} { By the proof of the above lemma, we known that for any $ \delta \in (0,1)$, there exists a  $\delta$-lattice $\{z_k\}$ such that $\inf_k|z_k|\geq r$, where $r=\frac{1-\sqrt{1-\delta^2}}{\delta}$.}\msk

The following lemma formulates a stability property of a separated sequence under a small  perturbation.\msk

\begin{lemma}[{\cite[Lemma 2.3]{LSS}}] \label{S}   Let $N>1$. Suppose that $\{a_k\}$ is a $r_0$-separated sequence of $\mathbb{D}$ with $|a_k|\geq 1-\frac{1}{2N}$ and $b_k=(1-N(1-|a_k|))a_k$. Then $\{b_k\}$ is a $\frac{\delta r_0}{N}$-separated sequence of $\mathbb{D}$, where $\delta$ is  a constant independent of $r_0$ and $N$.
\end{lemma}

\subsection{Area methods and equivalent norms}

For $\gamma>1$ and $\zeta\in\partial\D$, define the Koranyi(admissible, non-tangential) approach region $\Gamma_\gamma(\zeta)$ by
$$\Gamma_\gamma(\zeta) =\{z\in\D:|1-\overline{\zeta}z|<\frac{\gamma}{2}(1-|z|^2)\}.$$
In this paper we agree that $\Gamma(\zeta) =\Gamma_2(\zeta)$. It is known that for every $0<r<1$ and $\gamma>1$, there exists $\gamma^\prime>1$ so that
\begr\label{R}
\bigcup\limits_{z\in\Gamma_\gamma(\zeta)}\Delta(z,r)\subseteq\Gamma_{\gamma^\prime}(\zeta).
\endr

Given $z\in\D$,   define $$I(z)=\{\zeta\in\partial\D:z\in\Gamma(\zeta)\}\subset\partial\D.$$ Since $|I(z)|\asymp 1-|z|$, an application of Fubini's theorem yields the following important formula:
$$\int_{\D}\varphi(z)d\nu(z)\asymp\int_{\partial\D}\left(\int_{\Gamma(\zeta)}\varphi(z)\frac{d\nu(z)}{1-|z|}\right)dm(\zeta),$$
where $\varphi$ is any non-negative measurable function and $\nu$ is a finite positive measure.

The following result is the famous Calder\'{o}n theorem and can be found in \cite[Theorem 1.3]{Pa} or \cite[Theorem 5.3]{Pau}.

\begin{lemma} \label{eq} Suppose $0<p,q<\infty$, $f\in H^p$, $f(0)=0$. Then
$$\left(\int_{\partial\D}\left(\int_{\Gamma(\zeta)}\Delta |f|^q(z)dA(z)\right)^{\frac{p}{q}}dm(\zeta)\right)^{\frac{1}{p}}
\asymp\|f\|_{H^p}.$$
\end{lemma}

\subsection{Tent spaces}

Tent spaces were introduced by Coifman, Meyer and Stein \cite{CMS} in order to
study several problems in harmonic analysis. Luecking \cite{L} used these tent spaces to study embedding theorems for Hardy spaces on
$\mathbb{R}^n$, results that have been obtained in the unit ball $\mathbb{B}_n$ by Arsenovic and Jevtic \cite{A,J}.
Recently, Pel\'aez and R\"atty\"a \cite{PR} showed that tent spaces have natural analogues
for Bergman spaces, and they may play a role in the theory of small weighted Bergman
spaces similar to that of the original tent spaces in the Hardy space case.

For $0<p,q<\infty$ and a positive Borel measure $\nu$ on $\D$, finite on compact sets, the tent spaces $T^p_q(\nu)$ consists of the $\nu$-equivalence classes of $\nu$-measurable functions $f$ such that
$$\|f\|_{T^p_q(\nu)}^p=\int_{\T}\left(\int_{\Gamma(\zeta)}|f(z)|^qd\nu(z)\right)^{\frac{p}{q}}dm(\zeta)<\infty.$$
 For $0<p<\infty$, define
 $$\|f\|_{T^p_\infty(\nu)}^p=\int_{\T}\left(\nu-ess\sup_{z\in\Gamma(\zeta)}|f(z)|\right)^pdm(\zeta)<\infty.$$

For $0\neq a\in\D$, we define $\zeta_a=\frac{a}{|a|}$ and set
$$Q(a)=\{z\in\D: |1-\overline{\zeta_a}z|<1-|a|^2\}.$$
We also set $Q(0)=\D$.
The space $T^\infty_{q}(\nu)$ consists of $\nu$-measureable functions $f$ on $\D$ with
$$\|f\|_{T^\infty_{q}(\nu)}=\sup_{\zeta\in\partial\D}\left(\sup_{a\in\Gamma(\zeta)}\frac{1}{|I_a|}\int_{Q(a)}|f(z)|^q(1-|z|)d\nu(z)\right)^{\frac{1}{q}}<\infty,$$
where $I_a=\{\zeta\in\partial\D:|\arg\zeta-\arg a|\leq\frac{1-|a|}{2}\}.$
By Section 5.2 of \cite{Zhu}, we notice that $f\in T^\infty_{q}(\nu)$ if and only if
$|f(z)|^q(1-|z|)d\nu(z)$ is a Carleson measure on $\D$. The aperture $\gamma>0$ of Koranyi region is suppressed from the above notation. Moreover, it is well known that any two apertures generate the same function space with equivalent quasinorms.

   The following result gives a description of the dual of $T^{p}_{q}(\nu)$. See \cite{CMS, L, P}.

\begin{lemma}\label{dual}  Let $1\leq p,q<\infty$ with $p+q\neq2$, and let $\nu$ be a positive  Borel measure on $\D$, finite on compact sets of $\D$. Then
the dual of $T^p_q(\nu)$ is isomorphic to $T^{p^\prime}_{q^\prime}(\nu)$ under the pairing:
$$\langle f,g\rangle_{T^2_2(\nu)}=\int_{\D}f(z)\overline{g(z)}(1-|z|)d\nu(z),$$
and there exists a positive constant $C$ such that
\begr
\left|\langle f,g\rangle_{T^2_2(\nu)}\right|\leq C\|f\|_{T^p_q(\nu)}\|g\|_{T^{p^\prime}_{q^\prime}(\nu)}\nonumber
\endr
for any $f\in T^p_q(\nu)$ and $g\in T^{p^\prime}_{q^\prime}(\nu)$.
\end{lemma} \msk

If $\nu=\sum_k \delta_{z_k}(z)$(where $\delta_{z_k}$ denotes a unit mass at $z_k$), where $\{z_k\}$ is a separated sequence,  then we write $T^p_q(\nu)=T^p_q(\{z_k\})$.

   The following result gives a description of the dual of $T^p_q(\{z_k\})$. See \cite{A, J, L, PR}.\msk

\begin{lemma}\label{dual2}  Let $0<q\leq1<p<\infty$ and $\{z_k\}$ be a separated sequence on $\D$. Then
the dual of $T^p_q(\{z_k\})$ is isomorphic to $T^{p^\prime}_{\infty}(\{z_k\})$ under the pairing:
$$\langle f,g\rangle_{T^2_2(\{z_k\})}=\sum_k f(z_k)\overline{g(z_k)}(1-|z_k|),$$
and there exists a positive constant $C$ such that
\begr
\left|\langle f,g\rangle_{T^2_2(\{z_k\})}\right|\leq C\|f\|_{T^p_q(\{z_k\})}\|g\|_{T^{p^\prime}_{\infty}(\{z_k\})}\nonumber
\endr
for any $f\in T^p_q(\{z_k\})$ and $g\in T^{p^\prime}_{q^\prime}(\{z_k\})$.
\end{lemma}   \msk

We need the following important result, which plays in part a similar role on $H^p$ as the atomic decomposition on standard Bergman spaces $A^p_\alpha$.

\begin{lemma}[{\cite[Lemma D]{P}}] \label{SO}  Let $0<p<\infty$ and   $\{z_k\}$ be a separated sequence. Define
\begr
S_\lambda(f)(z)=\sum_{k}f(z_k)\left(\frac{1-|z_k|}{1-\overline{z_k}z}\right)^\lambda,~~~z\in\D.\nonumber
\endr
Then there exists $\lambda_0=\lambda_0(p)\geq1$ such that $S_\lambda:T^p_2(\{z_k\})\to H^p$ is bounded for all $\lambda>\lambda_0$.
\end{lemma} \msk

\subsection{Local estimates}

We shall also use the following inequality.   Here and henceforth $\Delta$ denotes the Laplacian operator.\msk

\begin{lemma}[{\cite[Lemma 3]{P}\label{L}}]  If $2\leq q<\infty$ and $0<r<1$, there is a constant $C(q,r)$ such that
\begr
|f^\prime(z)|^q(1-|z|^2)^q\leq C(q,r)\int_{\Delta(z,r)}\Delta|f|^q(\zeta)dA(\zeta),~~~z\in\D.\nonumber
\endr
\end{lemma} \msk

The following lemma plays an essential role in the proof of Theorem 1.

\begin{lemma}\label{LE} Let $0<p<\infty$,  and $0<s<r<1$. Then there exists a constant $C=C(p,s,r)$ such that
\begr
|f^\prime(z)-f^\prime(w)|^p\leq C\rho(z,w)^p\frac{\int_{\Delta(z,r)}|f(\zeta)|^pdA(\zeta)}{(1-|z|^2)^{p+2}}\nonumber
\endr
for all $z\in\D$, $w\in \Delta(z,s)$ and $f\in H(\D)$.
\end{lemma}\msk

\begin{proof} Let $0<s<r_1<r$, $r_2=r-r_1$. By \cite[Lemma 3.1]{S1} or \cite[Lemma 2.2]{KW}, we have
\begr
|f^\prime(z)-f^\prime(w)|^p\leq C\rho(z,w)^p\frac{\int_{\Delta(z,r_1)}|f^\prime(\eta)|^pdA(\eta)}{(1-|z|^2)^{2}} \label{DL}
\endr
for any $f\in H(\D)$.

On the other hand, by the mean value property and the fact that $\Delta(\eta,r_2)\subset\Delta(z,r)$ and $1-|z|^2\asymp1-|\eta|^2$ for any $\eta\in\Delta(z,r_1)$, we get
  \begr
|f^\prime(\eta)|^p&\lesssim&\frac{1}{(1-|\eta|^2)^{p+2}}\int_{\Delta(\eta,r_2)}|f(\zeta)|^pdA(\zeta)\nonumber\\
&\lesssim&\frac{1}{(1-|z|^2)^{p+2}}\int_{\Delta(z,r)}|f(\zeta)|^pdA(\zeta).\nonumber
\endr
Then the above estimate together with $A(\Delta(z,r_1))\asymp(1-|z|^2)^2$ yield the desired result.
\end{proof}

\begin{lemma}\label{Cs1} Let $0<p<s\leq q<\infty$  and $0<r<1$. Then
\begr
\int_{\D}|f^\prime(z)|^qd\mu(z)\lesssim\left(\int_{\D}|f(\zeta)|^s\frac{(\mu(\Delta(\zeta,r)))^{\frac{s}{q}}}{(1-|\zeta|^2)^{s+2}}dA(\zeta)\right)^{\frac{q}{s}}\nonumber
\endr
for all $f\in H(\D)$.
\end{lemma}

\begin{proof}
Recall the known estimate
$$|f^\prime(z)|^s\lesssim\frac{1}{(1-|z|^2)^{s+2}}\int_{\Delta(z,r)}|f(\zeta)|^sdA(\zeta),~~~\mbox{~~}z\in\D,~~s>0.$$
Then, the above inequality  and Minkowski's  inequality give
\begr
\int_{\D}|f^\prime(z)|^qd\mu(z)&\lesssim&\int_{\D}\left(\frac{1}{(1-|z|^2)^{s+2}}\int_{\Delta(z,r)}|f(\zeta)|^sdA(\zeta)\right)^{\frac{q}{s}}d\mu(z)\nonumber\\
&\lesssim&\left(\int_{\D}|f(\zeta)|^s\frac{(\mu(\Delta(\zeta,r)))^{\frac{s}{q}}}{(1-|\zeta|^2)^{s+2}}dA(\zeta)\right)^{\frac{q}{s}}.  \nonumber
\endr
\end{proof}

\begin{lemma} \label{Cs2}  Let $2\leq p<\infty$  and $0<r<1$. Then
\begr
\int_{\D}|f^\prime(z)|^pd\mu(z)\lesssim\int_{\D}\frac{\Delta|f|^p(\zeta)}{(1-|\zeta|)^p}\mu(\Delta(\zeta,r))dA(\zeta).\nonumber
\endr
\end{lemma}

\begin{proof} By Lemma \ref{L}, Fubini's theorem and the fact that $1-|z|\asymp1-|\zeta|$ for any $\zeta\in\Delta(z,r)$, we get the desired result.
\end{proof}

\begin{lemma}\label{Cs3}  Suppose $0<q<\infty$  and $0<r<1$. Then
\begr
\int_{\D}|f^\prime(z)|^qd\mu(z)\lesssim\int_{\D}|f^\prime(\zeta)|^q\mu(\Delta(\zeta,r))dh(\zeta),\nonumber
\endr
where $dh(z)=\frac{dA(z)}{(1-|z|^2)^2}$ denote the hyperbolic measure.
\end{lemma}

\begin{proof} By the mean value property and Fubini's theorem and the fact that $1-|z|^2\asymp1-|\zeta|^2$ for any $z\in\Delta(\zeta,r)$, we have
\begr
\int_{\D}|f^\prime(z)|^qd\mu(z)&\lesssim& \int_{\D}\frac{1}{(1-|z|^2)^2}\int_{\Delta(z,s)}|f^\prime(\zeta)|^qdA(\zeta)d\mu(z)\nonumber\\
&=&\int_{\D}|f^\prime(\zeta)|^q\int_{\Delta(\zeta,s)}\frac{1}{(1-|z|^2)^2}d\mu(z)dA(\zeta)\nonumber\\
&\lesssim&\int_{\D}|f^\prime(\zeta)|^q\mu(\Delta(\zeta,r))dh(\zeta).\nonumber
\endr
\end{proof}

  By \cite[Lemma 4.30]{Zhu}, for all $a,z,w\in\mathbb D$ with $\rho(z,w)<r$ and any real $s$, we have
$$\left|1-\left(\frac{1-\overline{a}z}{1-\overline{a}w}\right)^s\right|\leq C(s,r)\rho(z,w),$$
and therefore, for all $w,z,a\in\mathbb D$ with $z\in\triangle(a,r)$ and any $s>0$,
$$\left|\frac{1}{(1-\overline{a}z)^s}-\frac{1}{(1-\overline{a}w)^s}\right|\leq C(s,r)\rho(z,w)\left|\frac{1}{(1-\overline{a}z)^s}\right|.  $$
Although the converse inequality does not hold   we have the following   result, which can be found in \cite[Theorem 2.8]{KW} or \cite[Lemma D]{SQL}.
  %which is crucial in the proof of the necessary part of  Theorem 1.
  Since Koo and Wang's result  is stated for the unit ball setting (see \cite{KW}), for the convenience of  readers, and in order to offer no doubt of the validity of the result, we give a proof here.

\begin{lemma}[{\cite[Theorem 2.8]{KW}}] \label{KE}   Suppose $s>1$ and $0<r_0<1$. Then there are $N=N(r_0)>1$ and $C=C(s,r_0)$ such that
\begr
&&\left|\frac{1}{(1-\overline{a}z)^s}-\f{1}{(1-\overline{a}w)^s}\right|+
\left|\frac{1}{(1-t_N\overline{a}z)^s}-\f{1}{(1-t_N\overline{a}w)^s}\right|\nonumber\\
&\geq& C\rho(z,w)\left|\frac{1}{(1-\overline{a}z)^s}\right|, \nonumber
\endr
for all $z\in\triangle(a,r_0)$ with $1-|a|<\f{1}{2N}$, $t_N=1-N(1-|a|)$ and $w\in\D$.
\end{lemma}\msk

\begin{proof} For $a\in\D\backslash\{0\}$ and $0<r<1$, denote $S(a,r) =\{z\in\D:|1-\frac{\overline{a}}{|a|}z|<r\}$. Let $z\in\Delta(a,r_0)$.
Since
\begr
|1-\overline{a}z|=|(1-|a|^2)+\overline{a}(a-z)|<2(1-|a|)+r_0|1-\overline{a}z|,\nonumber
\endr
we have
\begr|1-\overline{a}z|<\frac{2}{1-r_0}(1-|a|).\label{f1}\endr
Take $N=(\frac{10}{1-r_0})^2$. We shall split the proof into two cases.\\
\par Case I. $ w\notin S(\frac{a}{|a|},\sqrt{N}(1-|a|))$. In this case, we have
$$|1-\overline{a}w|\geq|a|\left|1-\frac{\overline{a}}{|a|}w\right|-(1-|a|)\geq(\frac{\sqrt{N}}{2}-1)(1-|a|)\geq\frac{4}{1-r_0}(1-|a|).$$
Combine this with (\ref{f1}), we have
\begr
&&\left|\frac{1}{(1-\overline{a}z)^s}-\f{1}{(1-\overline{a}w)^s}\right|
 \geq \frac{1}{\left|1-\overline{a}z\right|^s}-\f{1}{\left|1-\overline{a}w\right|^s}\nonumber\\
&\geq&\left[\left(\frac{1-r_0}{2}\right)^s-\left(\frac{1-r_0}{4}\right)^s\right]\frac{1}{(1-|a|)^s}
 \geq  C(r_0,s)\rho(z,w)\left|\frac{1}{(1-\overline{a}z)^s}\right|. \nonumber
\endr
\par Case II. $w\in S(\frac{a}{|a|},\sqrt{N}(1-|a|))$. Set $u(a)=\frac{1-t_N\overline{a}z}{1-t_N\overline{a}w}$. If $|u(a)|\leq\frac{1}{2}$ or $|u(a)|\geq 2$,  by using elementary estimates  we easily get that
$$|1-u(a)^s|\gtrsim|1-u(a)|.$$
Now, assume that $\frac{1}{2}\leq|u(a)|\leq 2$. Denote $u(a)=re^{i\theta}$, where $\theta=\arg u(a)$.
Since
$$\left|1-\overline{a}w\right|\leq|a|\left|1-\frac{\overline{a}}{|a|}w\right|+(1-|a|)
<2\sqrt{N}(1-|a|),$$
we have
$$\left|(1-t_N\overline{a}w)-N(1-|a|)\right|\leq|1-\overline{a}w|\leq2\sqrt{N}(1-|a|).$$
By (\ref{f1}), we have
\begr
\left|(1-t_N\overline{a}z)-N(1-|a|)\right|\leq|1-\overline{a}z|\leq\sqrt{N}(1-|a|). \label{f2}
\endr
Thus, $1-t_N\overline{a}z$ and $1-t_N\overline{a}w$ are points inside the disc centered at $N(1-|a|)$ with radius $2\sqrt{N}(1-|a|)$. Fix $0<\epsilon<1$ small enough and choose $N(r_0)$ sufficiently large such that
$$|\arg(1-t_N\overline{a}z)|<\frac{\epsilon}{2}\mbox{~~~,~~}|\arg(1-t_N\overline{a}w)|<\frac{\epsilon}{2}$$
and
$$|\arg\left(1-t_N\overline{a}z\right)^s|<\frac{\epsilon}{2}\mbox{~~~,~~~}|\arg\left(1-t_N\overline{a}w\right)^s|<\frac{\epsilon}{2}.$$
Then $|\theta|<\epsilon$ and $|s\theta|<\epsilon$. Since $\epsilon$ is small enough, we have
\begr
|1-u(a)^s|&\asymp&|1-r^s\cos(s\theta)|+r^s|\sin(s\theta)|\nonumber\\
&=&|1-r^s|+r^s(|\sin(s\theta)|+(1-\cos(s\theta)))\nonumber\\
&\asymp&|1-r^s|+sr^s|\theta|\nonumber\\
&\asymp&|1-r|+|\theta|\nonumber\\
&=&|1-r\cos\theta|+r(|\theta|-(1-\cos\theta))\nonumber\\
&\asymp&|1-r\cos\theta|+r(|\sin\theta|)\nonumber\\
&\asymp&|1-u(a)|.\nonumber
\endr
Therefore
\begr
&&\left|\frac{1}{(1-t_N\overline{a}z)^s}-\f{1}{(1-t_N\overline{a}w)^s}\right|\nonumber\\
&=&\frac{1}{|1-t_N\overline{a}z|^s}|1-u(a)^s| \gtrsim \frac{1}{|1-\overline{a}z|^s}|1-u(a)|\nonumber\\
&\gtrsim&\frac{1}{|1-\overline{a}z|^s}\frac{t_N|a|(z-w)|}{|1-t_N\overline{a}w|} \gtrsim \frac{1}{|1-\overline{a}z|^s}\frac{|z-w|}{|1-\overline{z}w|},\nonumber
\endr
where we used the fact that
$$|1-t_N\overline{a}z|\lesssim|1-\overline{a}z|,|1-t_N\overline{a}w|\lesssim|1-\overline{a}w|~~~\mbox{~~~and~~~}|1-\overline{a}w|\asymp|1-\overline{z}w|.$$
The proof is complete. \end{proof}

\section{  $D_{\varphi}-D_{\psi}$ from $H^p$ to $A^q_\alpha$}

We will split Theorem \ref{main} into two theorems, i.e., Theorem 14 and Theorem 15.

\begin{thm} \label{Th2} Suppose $\varphi$ and $\psi$ are analytic self-maps of $\D$.
 Let either $0<p<q< \infty$ or $2\leq p=q$, $0<r<1$ and $\mu$ denote the joint pull-back measure  induced by $\varphi$ and $\psi$ defined as (\ref{m}).

(a) $D_\varphi-D_\psi:H^p\to A^q_\alpha$ is bounded if and only if
\begr
\sup_{z\in\D}\frac{\mu(\Delta(z,r))}{(1-|z|^2)^{\frac{q}{p}+q}}<\infty.
\nonumber\endr
Moreover,
$$\|D_\varphi-D_\psi\|_{H^p\to A^q_\alpha}\asymp
\sup_{z\in\D}\frac{\mu(\Delta(z,r))}{(1-|z|^2)^{\frac{q}{p}+q}}.
\nonumber$$

(b) $D_\varphi-D_\psi:H^p\to A^q_\alpha$ is compact if and only if
\begr\lim_{|z|\to1}\frac{\mu(\Delta(z,r))}{(1-|z|^2)^{\frac{q}{p}+q}}=0.
\nonumber\endr
\end{thm}

\begin{proof}
(a) First, consider the lower bound. Suppose that $D_\varphi-D_\psi: H^p\to A^q_\alpha$ is bounded. Let $$k_a(z)=\big(\frac{1-|a|^2}{(1-\overline{a}z)^2}\big)^{\frac{1}{p}}, ~~~~~~~ k_{a_N}(z)=\big(\frac{1-|a|^2}{(1-t_N\overline{a}z)^2}\big)^{\frac{1}{p}},$$  where $|a|>r_1=1-\frac{1}{2N}$, $t_N$ and $N$ are defined as Lemma \ref{KE}. By Lemma \ref{KE}, we get
\begr
&&\|(D_\varphi-D_\psi)k_a\|_{A^q_\alpha}^q+\|(D_\varphi-D_\psi)k_{a_N}\|_{A^q_\alpha}^q\nonumber\\
&\gtrsim&\int_{\varphi^{-1}(\Delta(a,r))}\frac{\rho(\xi)^q}{(1-|a|^2)^{\frac{q}{p}+q}}dA_\alpha(\xi)\label{2.1}
\endr
and
\begr
&&\|(D_\varphi-D_\psi)k_a\|_{A^q_\alpha}^q+\|(D_\varphi-D_\psi)k_{a_N}\|_{A^q_\alpha}^q\nonumber\\
&\gtrsim&\int_{\psi^{-1}(\Delta(a,r))}\frac{\rho(\xi)^q}{(1-|a|^2)^{\frac{q}{p}+q}}dA_\alpha(\xi).\label{2.2}
\endr
Therefore,
\begr
\|D_\varphi-D_\psi\|^q&\gtrsim&\sup_{a\in\D}\left(\|(D_\varphi-D_\psi)k_a\|_{A^q_\alpha}^q+\|(D_\varphi-D_\psi)k_{a_N}\|_{A^q_\alpha}^q\right)\nonumber\\
&\gtrsim&
\sup_{|a|>r_1}\frac{\mu(\Delta(a,r))}{(1-|a|^2)^{\frac{q}{p}+q}}.  \nonumber
\endr
For $|a|\leq r_1$, take $r_2=\f{r+r_1}{1+rr_1}$, then
$\triangle(a,r)\subset\triangle(0,r_2)$. Therefore
\begr
&&\frac{\mu(\triangle(a,r))}{(1-|a|^2)^{q/p+q}}\nonumber\\
&\leq& \frac{1}{(1-r_1^2)^{q/p+q}} \left(\int_{\varphi^{-1}(\triangle(a,r))}\rho(z)^qdA_\alpha(z)+\int_{\psi^{-1}(\triangle(a,r))}\rho(z)^qdA_\alpha(z)\right)\nonumber\\
&\leq& \frac{1}{(1-r_1^2)^{q/p+q}}\left(\int_{\varphi^{-1}(\triangle(a,r))}+\int_{\psi^{-1}(\triangle(a,r))}\right)\f{|\varphi(z)-\psi(z)|^q}{(1-r_2)^q}dA_\alpha(z)\nonumber\\
&\lesssim& \|(D_\varphi-D_\psi)(z^2)\|_{A^q_\alpha}^q\nonumber\\
&\lesssim& \|D_\varphi-D_\psi\|^q.\label{cp}
\endr

Next we will consider the upper bound. For $f\in H^p$, consider
\begr
\|(D_\varphi-D_\psi)f\|_{A^q_\alpha}^q
&=&\big(\int_{\rho(z)\geq \frac{1}{2}}+\int_{\rho(z)<\frac{1}{2}}\big)|f^\prime\circ\varphi(z)-f^\prime\circ\psi(z)|^qdA_\alpha(z)\nonumber\\
&:=&I_1+I_2.\label{A}
\endr
The first term is uniformly bounded.
\begr
I_1
\leq2^{2q}\int_{\D}|f^\prime(z)|^qd\mu(z). ~~~~~\, ~~~~~\, ~~~~~\, ~~~~~\, ~~~~~\, ~~~~~\, ~~~~~\, ~~~~~\, ~~~~~\, ~~~~~\, ~~~~~\, ~~~~~\, ~~~~~\, ~~~~~\, ~~~~~\, ~~~~~\, \label{I1}
\endr

We shall split the proof into two cases.

\par {\bf Case $0<p<q<\infty$.}  Let $0<r<1$ be fixed. Choose $s\in(p,q]$. By Lemma \ref{LE}, we get
 \begr
I_2
&\lesssim&\int_{\rho(z)<\frac{1}{2}}|\rho(z)|^q\left(\int_{\Delta(\varphi(z),r)}\frac{|f(w)|^s}{(1-|w|^2)^{s+2}}dA(w)\right)^{\frac{q}{s}}dA_\alpha(z). \nonumber
\endr
So, by  Minkowski's inequality we obtain
\begr
I_2
&\lesssim&\int_{\D}\rho(z)^q\left(\int_{\Delta(\varphi(z),r)}\frac{|f(w)|^s}{(1-|w|^2)^{s+2}}dA(w)\right)^{\frac{q}{s}}dA_\alpha(z)\nonumber\\
&\lesssim&\left(\int_{\D}|f(w)|^s\frac{\left(\int_{\varphi^{-1}(\Delta(w,r))}\rho(z)^qdA_\alpha(z)\right)^{\frac{s}{q}}}{(1-|w|^2)^{s+2}}dA(w)\right)^\frac{q}{s}\nonumber\\
&\leq&\left(\int_{\D}|f(w)|^s\frac{\left(\mu(\Delta(w,r))\right)^{\frac{s}{q}}}{(1-|w|^2)^{s+2}}dA(w)\right)^\frac{q}{s}. \label{Cs1I2}
\endr
Therefore, by (\ref{A}), (\ref{I1}), (\ref{Cs1I2}) and Lemma \ref{Cs1}, we get
\begr
&& \|(D_\varphi-D_\psi)f\|_{A^q_\alpha}^q\nonumber\\
&\lesssim&\left(\int_{\D}|f(w)|^s\frac{\left(\mu(\Delta(w,r))\right)^{\frac{s}{q}}}{(1-|w|^2)^{s+2}}dA(w)\right)^\frac{q}{s}\label{C1}\\
&\lesssim&\sup_{z\in\D}\frac{\mu(\Delta(z,r))}{(1-|z|^2)^{\frac{q}{p}+q}}\left(\int_{\D}|f(w)|^s(1-|w|^2)^{\frac{s}{p}-2}dA(w)\right)^\frac{q}{s}. \nonumber
\endr
Since $s>p$, by Duren's theorem we have $\|f\|_{A^s_{\frac{s}{p}-2}}\lesssim\|f\|_{H^p}.$
Therefore
$$\|D_\varphi-D_\psi\|^q\lesssim\sup_{z\in\D}\frac{\mu(\Delta(z,r))}{(1-|z|^2)^{\frac{q}{p}+q}}.$$
This completes the proof of the case $0<p<q<\infty.$\\

  {\bf Case $q=p\geq2$.} By (\ref{DL}), Lemma \ref{L} and Fubini's theorem, we get
{\small \begr
&&I_2= \int_{\rho(z)<\frac{1}{2}}|f^\prime\circ\varphi(z)-f^\prime\circ\psi(z)|^qdA_\alpha(z)\nonumber\\
&\lesssim&\int_{\rho(z)<\frac{1}{2}}\rho(z)^q\left(\int_{\Delta(\varphi(z),s_1)}\frac{|f^\prime(w)|^q}{(1-|w|^2)^{2}}dA(w)\right)dA_\alpha(z)\label{I2}\\
&\lesssim&\int_{\D}\frac{\rho(z)^q}{(1-|\varphi(z)|^2)^{2+q}}\left(\int_{\Delta(\varphi(z),s_1)}\left(\int_{\Delta(w,s_2)}\Delta|f|^q(u)dA(u)\right)dA(w)\right) dA_\alpha(z)\nonumber\\
&\lesssim&\int_{\D}\frac{\rho(z)^q}{(1-|\varphi(z)|^2)^{q}}\left(\int_{\Delta(\varphi(z),r)}\Delta|f|^q(u)dA(u)\right) dA_\alpha(z)\nonumber\\
&\lesssim&\int_{\D}\frac{\Delta|f|^q(u)}{(1-|u|^2)^{q}}\mu(\Delta(u,r))dA(u), \nonumber
\endr  }
where $s_1, s_2\in(0,1)$ are chosen sufficiently small depending only on $r$, for example, we can take $s_1=s_2=\frac{1-\sqrt{1-r^2}}{r}$.
From this with (\ref{A}), (\ref{I1}) and  Lemma \ref{Cs2}, we have
 \begr
\|(D_\varphi-D_\psi)f\|_{A^q_\alpha}^q\lesssim\int_{\D}\frac{\Delta|f|^q(u)}{(1-|u|^2)^{q}}\mu(\Delta(u,r))dA(u).\label{C2}
\endr
Using the Hardy-Stein-Spencer identity
 $$\|f||_{H^p}^p=|f(0)|^p+\frac{1}{2}\int_{\D}\Delta|f(z)|^p\log\frac{1}{|z|}dA(z),$$ and the fact that $p=q$,
 the proof can be finished as in the previous case.\msk

(b)
We now turn to the proof of the compactness.

If $D_\varphi-D_\psi:H^p\to A^q_\alpha$ is compact, then one may deduce from (\ref{2.1}), (\ref{2.2}) that
$$\lim_{|a|\to1}\frac{\mu(\Delta(a,r))}{(1-|a|^2)^{\frac{q}{p}+q}}=0.$$
Now, suppose
$$\lim_{|a|\to1}\frac{\mu(\Delta(a,r))}{(1-|a|^2)^{\frac{q}{p}+q}}=0.$$
To prove the compactness of $D_\varphi-D_\psi:H^p\to A^q_\alpha$, we consider an arbitrary sequence $\{f_n\}$ in $H^p$ such that $\|f_n\|_{H^p}\leq1$ and $f_n\to0$ uniformly on compact subsets of $\D$. It is enough to show that
\begr\lim_{n\to\infty}\|(D_\varphi-D_\psi)f_n\|_{A^q_\alpha}=0.\label{T}\endr
We first consider the case $0<p<q<\infty$. Let $t\in(0,1)$.   By (\ref{C1}) we have
\begr
&&\|(D_\varphi-D_\psi)f_n\|_{A^q_\alpha}^q\nonumber\\
&\lesssim&\left(\int_{t\D}+\int_{\D\backslash t\D}|f_n(w)|^s\frac{\left(\mu(\Delta(w,r))\right)^{\frac{s}{q}}}{(1-|w|^2)^{s+2}}dA(w)\right)^\frac{q}{s}.
\endr
Note that $\mu$ is a finite measure,   thus
$$\sup_{a\in t\D}\frac{\mu(\Delta(a,r))}{(1-|a|^2)^{\frac{q}{p}+q}}<\infty.$$
Now, since $f_n\to0$ uniformly on $t\D$, we obtain
$$\lim_{n\to\infty}\int_{t\D}|f_n(w)|^s\frac{\left(\mu(\Delta(w,r))\right)^{\frac{s}{q}}}{(1-|w|^2)^{s+2}}dA(w)=0.$$
On the other hand, we have
\begr
&&\left(\int_{\D\backslash t\D}|f_n(w)|^s\frac{\left(\mu(\Delta(w,r))\right)^{\frac{s}{q}}}{(1-|w|^2)^{s+2}}dA(w)\right)^\frac{q}{s}\nonumber\\
&\lesssim&\sup_{z\in\D\backslash t\D}\frac{\mu(\Delta(z,r))}{(1-|z|^2)^{\frac{q}{p}+q}}\left(\int_{\D}|f_n(w)|^s(1-|w|^2)^{\frac{s}{p}-2}dA(w)\right)^\frac{q}{s}\nonumber\\
&\lesssim&\sup_{z\in\D\backslash t\D}\frac{\mu(\Delta(z,r))}{(1-|z|^2)^{\frac{q}{p}+q}}. \nonumber
\endr
Letting $t\to1$, we get (\ref{T}), as desired.

The case $2<p=q<\infty$. By an similar argument, using (\ref{C2}) and the Hardy-Stein-Spencer identity, we deduce (\ref{T}), as desired.
\end{proof}
\msk

\begin{thm}\label{Th3} Suppose $\varphi$ and $\psi$ are analytic self-maps of $\D$.  Let $0<r<1$, $0<q\leq p< \infty$, and let $\mu$ denote the joint pull-back measure  induced by $\varphi$ and $\psi$ defined as (\ref{m}).\\
(i)~If $q<\min\{2,p\}$, then the following conditions are equivalent:

(ia) $D_\varphi-D_\psi:H^p\to A^q_\alpha$ is bounded;

(ib) $D_\varphi-D_\psi:H^p\to A^q_\alpha$ is compact;

(ic) The function \begr
\zeta\mapsto\left(\int_{\Gamma(\zeta)}\left(\frac{\mu(\Delta(z,r))}{(1-|z|^2)^{1+q}}\right)^{\frac{2}{2-q}}\frac{dA(z)}{(1-|z|)^2}\right)^{\frac{2-q}{2}}
\nonumber
\endr
\par belongs to $L^{\frac{p}{p-q}}(\partial\D,m)$.\msk\\
(ii)~~~If~$0<q=p<2$, then

(iia) $D_\varphi-D_\psi:H^p\to A^q_\alpha$ is bounded if and only if
\begr \left(\frac{\mu(\Delta(z,r))}{(1-|z|^2)^{1+p}}\right)^{\frac{2}{2-p}}\frac{dA(z)}{1-|z|} ~~ \nonumber \endr
\par is a Carleson measure.\msk

(iib) $D_\varphi-D_\psi:H^p\to A^q_\alpha$ is compact if and only if
\begr \left(\frac{\mu(\Delta(z,r))}{(1-|z|^2)^{1+p}}\right)^{\frac{2}{2-p}}\frac{dA(z)}{1-|z|} ~~ \nonumber \endr
\par is a  vanishing Carleson measure.\msk\\
(iii)~If $2\leq q<p<\infty$, then

(iiia) $D_\varphi-D_\psi:H^p\to A^q_\alpha$ is bounded if and only if the function
 \begr
\zeta\mapsto\sup_{z\in\Gamma(\zeta)}\frac{\mu(\Delta(z,r))}{(1-|z|^2)^{1+q}}
\nonumber
\endr
\par belongs to $L^{\frac{p}{p-q}}(\partial\D,m)$.

(iiib) $D_\varphi-D_\psi:H^p\to A^q_\alpha$ is compact if and only if the function
 \begr
\zeta\mapsto\sup_{z\in\Gamma(\zeta)\cap\{|z|\geq r\}}\frac{\mu(\Delta(z,r))}{(1-|z|^2)^{1+q}}
\nonumber
\endr
\par converges to zero in $L^{\frac{p}{p-q}}(\partial\D,m)$ as $r\to1^-$.
\end{thm}

\begin{proof}  By Lemma \ref{e}, there exists a $\delta$-lattice $\{z_k\}$ such that $\inf_{k}|z_k|>1-\frac{1}{2N}$, where $N$ is defined  as Lemma \ref{KE}.  Let $$\mathcal{C}T^p_2(\{z_k\})=\{f\in T^p_2(\{z_k\}):\|f\|_{T^p_2(\{z_k\})}=1\}.$$
  For each $\lambda>\lambda_0$ ($\lambda_0$ is that of Lemma \ref{SO} ), set
$$S_{\lambda}(f)(z)=\sum_{k=1}^\infty f(z_k)\left(\frac{1-|z_k|}{1-\overline{z_k}z}\right)^\lambda,~~\mbox{~~}f\in\mathcal{C}T^p_2(\{z_k\}),~~z\in\D,$$
and
$$S_{\lambda,N}(f)(z)=\sum_{k=1}^\infty f(z_k)\left(\frac{1-|z_k|}{1-t_{k,N}\overline{z_k}z}\right)^\lambda,~~\mbox{~~}f\in\mathcal{C}T^p_2(\{z_k\}),~~z\in\D, $$
where $t_{k,N}=1-N(1-|z_k|)$.
By Lemmas \ref{S} and  \ref{SO}, we have
$$\|S_{\lambda}(f)\|_{H^p}\lesssim\|f\|_{T^p_2(\{z_k\})}~~~\mbox{~~and~~}\|S_{\lambda,N}(f)\|_{H^p}\lesssim\|f\|_{T^p_2(\{z_k\})}.$$
Therefore,
\begr
\int_{\D}\left|(D_\varphi-D_\psi)S_{\lambda}(f)(z)\right|^qdA_\alpha(z)\lesssim\|D_\varphi-D_\psi\|^q\|f\|_{T^p_2(\{z_k\})}^q.\nonumber
\endr
Let $g_t\in T^p_2(\{z_k\})$ such that $g_t(z_k)=f(z_k)r_k(t)$, where $r_k$ denotes the $k$th Rademacher function. Replace  $f(z_k)$ by $g_t(z_k)$ in the above inequality, and integrate with respect to $t$ we obtain
\begr
\int_0^1\int_{\D}\left|(D_\varphi-D_\psi)S_{\lambda}(g_t)(z)\right|^qdA_\alpha(z)dt\lesssim\|D_\varphi-D_\psi\|^q\|f\|_{T^p_2(\{z_k\})}^q.\nonumber
\endr
Using Fubini's theorem and   Khinchine's inequality we get
\begr
I
&=&\int_{\D}\left(\sum_{k=1}^\infty |f(z_k)|^2\left|\frac{(1-|z_k|)^{\lambda}}{(1-\overline{z_k}\varphi(z))^{\lambda+1}}-\frac{(1-|z_k|)^{\lambda}}{(1-\overline{z_k}\psi(z))^{\lambda+1}}\right|^2 \right)^{\frac{q}{2}}dA_\alpha(z)\nonumber\\
&\lesssim&\|D_\varphi-D_\psi\|^q\|f\|_{T^p_2(\{z_k\})}^q. \label{I}
\endr
Replace now $S_\lambda(f)$ by $S_{\lambda,N}(f)$, we have
{\small \begr
&&II=\int_{\D}\left(\sum_{k=1}^\infty |f(z_k)|^2\left|\frac{(1-|z_k|)^{\lambda}}{(1-t_{k,N}\overline{z_k}\varphi(z))^{\lambda+1}}-\frac{(1-|z_k|)^{\lambda}}{(1-t_{k,N}\overline{z_k}\psi(z))^{\lambda+1}}\right|^2 \right)^{\frac{q}{2}}dA_\alpha(z)\nonumber\\
&\lesssim&\|D_\varphi-D_\psi\|^q\|f\|_{T^p_2(\{z_k\})}^q. \label{II}
\endr}
Now, for any fixed $r_0\in(0,1)$, by Lemma \ref{KE},
\begr
& & I+II \nonumber \\
&\gtrsim &\int_{\D}\Bigg(\sum_{k=1}^\infty |f(z_k)|^2\Big(\big|\frac{(1-|z_k|)^{\lambda}}{(1-\overline{z_k}\varphi(z))^{\lambda+1}}-\frac{(1-|z_k|)^{\lambda}}{(1-\overline{z_k}\psi(z))^{\lambda+1}}\big|^2
\nonumber\\& &~~~~~\mbox{~~~}~~~~~~~~+\big|\frac{(1-|z_k|)^{\lambda}}{(1-t_{k,N}\overline{z_k}\varphi(z))^{\lambda+1}}-\frac{(1-|z_k|)^{\lambda}}{(1-t_{k,N}\overline{z_k}\psi(z))^{\lambda+1}}\big|^2\Big) \Bigg)^{\frac{q}{2}}dA_\alpha(z)\nonumber\\
&\gtrsim& \int_{\D}\left(\sum_{k=1}^\infty |f(z_k)|^2\frac{\rho(z)^2\big(\chi_{\varphi^{-1}(\triangle(z_k,r_0))}(z)+\chi_{\psi^{-1}(\triangle(z_k,r_0))}(z)\big)}{(1-|z_k|)^2}\right)^{\frac{q}{2}}dA_\alpha(z).  \nonumber
\endr
If $q\geq2$, using the inequality $\sum_{j}a_j^x\leq\left(\sum_{j}a_j\right)^x$, valid for all $a_j\geq0$ and $x\geq1$, and if $0<q<2$, by H\"older's inequality and the fact that
there exists a constant $K$ such that every $z\in\D$ belongs to at most $K$ discs  $\Delta(z_k, r_0)$,
we get
\begr
&&
  I+II \nonumber \\
&\gtrsim &\int_{\D}\sum_{k=1}^\infty |f(z_k)|^q\frac{\rho(z)^q\big(\chi_{\varphi^{-1}(\triangle(z_k,r_0))}(z)+\chi_{\psi^{-1}(\triangle(z_k,r_0))}(z)\big)^{\frac{q}{2}}}{(1-|z_k|)^{q}}dA_\alpha(z)\nonumber\\
&\gtrsim&\sum_{k=1}^\infty |f(z_k)|^q\frac{\mu(\triangle(z_k,r_0))}{(1-|z_k|)^{q}}.\nonumber
\endr
Therefore
\begr
&&\sum_{k=1}^\infty |f(z_k)|^q\frac{\mu(\triangle(z_k,r_0))}{(1-|z_k|)^{q}}\lesssim\|D_\varphi-D_\psi\|^q\|f\|_{T^p_2(\{z_k\})}^q\nonumber\\
&=&\|D_\varphi-D_\psi\|^q\left(\int_{\partial\D}\left(\left(\sum_{z_k\in\Gamma(\zeta)}(|f(z_k)|^q)^{\frac{2}{q}}\right)^{\frac{q}{2}}\right)^{\frac{p}{q}}dm(\zeta)\right)^{\frac{q}{p}}.\label{IAII}
\endr

On the other hand, by (\ref{A}), (\ref{I1}), (\ref{I2}) and Lemma \ref{Cs3}, we have

\begr
&&
\|(D_\varphi-D_\psi)f\|_{A^q_\alpha}^q\nonumber\\
&\lesssim&\int_{\D}|f^\prime(w)|^q\mu(\Delta(w,r))dh(w)\nonumber\\
&=&\int_{\D}\left(|f^\prime(w)|(1-|w|^2)\right)^q\frac{\mu(\Delta(w,r))}{(1-|w|^2)^{1+q}}(1-|w|)dh(w),\label{T2}
\endr
for any $0<r<1$.\msk

We now treat different cases separately.\msk
\par (i) It is trivial that $(ib)\implies(ia)$. Now, we prove that $(ia) \implies (ic)$. Since $0<q<\min\{2,p\}$, $s=\frac{p}{q}>1$ and $t=\frac{2}{q}>1$,    Lemma \ref{dual} yields $T^s_t(\{z_k\})^\star\simeq T^{s^\prime}_{t^\prime}(\{z_k\})$. Therefore, by (\ref{IAII}) we get
\begr
\left(\int_{\partial\D}\left(\sum_{z_k\in\Gamma(\zeta)}\big(\frac{\mu(\Delta(z_k,r_0))}{(1-|z_k|^2)^{1+q}}\big)^{\frac{2}{2-q}}\right)^{\frac{2-q}{2}\frac{p}{p-q}}dm(\zeta)\right)^{\frac{p-q}{p}}\lesssim\|D_\varphi-D_\psi\|^q.\nonumber
\endr
Using (\ref{R}) and the fact that $\Delta(z,r)\subset\Delta(z_k,r_0)$ for all $z\in\Delta(z_k,\delta)$ and all $k$, where $\frac{r+\delta}{1+r\delta}\leq r_0<1$, we get
\begr
&&\int_{\Gamma(\zeta)}\left(\frac{\mu(\Delta(z,r))}{(1-|z|^2)^{1+q}}\right)^{\frac{2}{2-q}}\frac{dA(z)}{(1-|z|)^2}\nonumber\\
&\leq&\sum_{k:\Delta(z_k,\delta)\cap\Gamma(\zeta)\neq\emptyset}\int_{\Delta(z_k,\delta)}
\left(\frac{\mu(\Delta(z,r))}{(1-|z|^2)^{1+q}}\right)^{\frac{2}{2-q}}\frac{dA(z)}{(1-|z|)^2}\nonumber\\
&\lesssim&\sum_{z_k\in\Gamma_\gamma(\zeta)}\left(\frac{1}{(1-|z_k|^2)^{1+q}}\right)^{\frac{2}{2-q}}\int_{\bigtriangleup(z_k,\delta)}\left(\mu(\Delta(z,r))\right)^{\frac{2}{2-q}}\frac{dA(z)}{(1-|z|)^2}\nonumber\\
&\lesssim&\sum_{z_k\in\Gamma_\gamma(\zeta)}\left(\frac{\mu(\Delta(z_k,r_0))}{(1-|z_k|^2)^{1+q}}\right)^{\frac{2}{2-q}},  \nonumber
\endr
and hence
\begr
&&\left(\int_{\partial\D}\Big(\int_{\Gamma(\zeta)}\big(\frac{\mu(\Delta(z,r))}{(1-|z|^2)^{1+q}}\big)^{\frac{2}{2-q}}\frac{dA(z)}{(1-|z|)^2}\Big)^{\frac{2-q}{2}\frac{p}{p-q}}dm(\zeta)\right)^{\frac{p-q}{p}}\nonumber\\
&\lesssim&\|D_\varphi-D_\psi\|^q.\nonumber
\endr
Then the assertion follows.

 Finally, let us prove that $(ic) \implies (ib)$. Let $ f_n \in H^p$ such that $\|f_n\|_{H^p}\leq1$ and $f_n\to0$ uniformly on compact subsets of $\D$.
 It is enough to show that
\begr\lim_{n\to\infty}\|(D_\varphi-D_\psi)f_n\|_{A^q_\alpha}=0.\label{C}\endr
 Denote $$F_n(w)=|f_n^\prime(w)|^q(1-|w|^2)^q, ~~dh_{R}=dh\chi_{\{R<|z|<1\}}, 0\leq R<1$$ and $\Phi_{\mu}(w)=\frac{\mu(\Delta(w,r))}{(1-|w|)^{1+q}}$. Fix $\epsilon>0$. Since $\Phi_{\mu}\in T_{\frac{2}{2-q}}^{\frac{p}{p-q}}(h)$, by the dominated convergence theorem there exists $R_0$ such that
 $$\sup_{R\geq R_0}\|\Phi_{\mu}\|_{T_{\frac{2}{2-q}}^{\frac{p}{p-q}}(h_R)}<\epsilon^q.$$
 Next, choose $k_0$ such that
 $$\sup_{n\geq k_0}\sup_{|z|\leq R_0}|f_n^\prime(z)|\leq\epsilon.$$
 Then, bearing in mind (\ref{T2}), using Lemma \ref{dual} and the inequality $$\|F_n\|_{T^{\frac{p}{q}}_{\frac{2}{q}}(h)}\lesssim\|f_n\|_{H^p}^q,$$
  we get
 \begr
&&
\|(D_\varphi-D_\psi)f_n\|_{A^q_\alpha}^q\nonumber\\
&\lesssim&\epsilon^q\int_{R_0\D}(1-|w|)^q\Phi_{\mu}(w)(1-|w|)dh(w)\nonumber\\
&~~&+\int_{\D}F_n(w)\Phi_{\mu}(w)(1-|w|)dh_{R_0}(w)\nonumber\\
&\lesssim&\epsilon^q\left\langle(1-|w|)^q,\Phi_{\mu}\right \rangle_{T^2_2(h)}+\left\langle F_n, \Phi_{\mu}\right \rangle_{T^2_2(h_{R_0})}\nonumber\\
&\lesssim&\epsilon^q\|(1-|w|)^q\|_{T^{\frac{p}{q}}_{\frac{2}{q}}(h)}\|\Phi_{\mu}\|_{T_{\frac{2}{2-q}}^{\frac{p}{p-q}}(h)}+\|F_n\|_{T^{\frac{p}{q}}_{\frac{2}{q}}(h)} \|\Phi_{\mu}\|_{T_{\frac{2}{2-q}}^{\frac{p}{p-q}}(h_{R_0})}\nonumber\\
&\lesssim&\epsilon^q.\nonumber
\endr
Letting $\epsilon\to0$, we get the desired result.\msk

\par (ii) (iia) Suppose that $D_\varphi-D_\psi:H^p\to A^q_\alpha$ is bounded.
Since $0<q=p<2$, then $s=\frac{p}{q}=1$ and $t=\frac{2}{p}>1$. So by Lemma \ref{dual}, $(T^1_t(\{z_k\}))^*\simeq T^\infty_{t^\prime}(\{z_k\})$ with the equivalence   norms. Therefore (\ref{IAII}) yields
\begr&&\sup_{\zeta\in\partial\D}\sup_{a\in\Gamma(\zeta)}\left(\frac{1}{|I_a|}\sum_{z_k\in Q(a)}\left(\frac{\mu(\Delta(z_k,r_0))}{(1-|z_k|)^{1+p}}\right)^{\frac{2}{2-p}}(1-|z_k|)\right)^{\frac{2-p}{2}}\nonumber\\
&\lesssim&\|D_\varphi-D_\psi\|^q,\nonumber\endr
that is
\begr
\sup_{a\in\D}\left(\frac{1}{|I_a|}\sum_{z_k\in Q(a)}\left(\frac{\mu(\Delta(z_k,r_0))}{(1-|z_k|)^{1+p}}\right)^{\frac{2}{2-p}}|I_{z_k}|\right)^{\frac{2-p}{2}}\lesssim\|D_\varphi-D_\psi\|^q.\label{All}
\endr
The point $z\in\D$ for which $Q(a)\cap\Delta(z,r)\neq\emptyset$ is contained in some $Q(a^\prime)$, where $\arg a=\arg a^\prime$ and $1-|a^\prime|\asymp1-|a|$, for all $|a|>R$, where $R=R(r)\in(0,1)$.
 Then
\begr
&&\int_{Q(a)}\left(\frac{\mu(\Delta(z,r))}{(1-|z|^2)^{1+p}}\right)^{\frac{2}{2-p}}\frac{dA(z)}{1-|z|}\nonumber\\
&\leq&\sum_{k:\Delta(z_k,\delta)\cap Q(a)\neq\emptyset}\int_{\Delta(z_k,\delta)}\left(\frac{\mu(\Delta(z,r))}{(1-|z|^2)^{1+p}}\right)^{\frac{2}{2-p}}\frac{dA(z)}{1-|z|}\nonumber\\
&\lesssim&\sum_{z_k\in Q(a^\prime)}\left(\frac{1}{(1-|z_k|^2)^{1+p}}\right)^{\frac{2}{2-p}}\int_{\Delta(z_k,\delta)}\left(\mu(\Delta(z,r))\right)^{\frac{2}{2-p}}\frac{dA(z)}{1-|z|}\nonumber\\
&\lesssim&\sum_{z_k\in Q(a^\prime)}\left(\frac{\mu(\Delta(z_k,r_0))}{(1-|z_k|^2)^{1+p}}\right)^{\frac{2}{2-p}}(1-|z_k|),\nonumber
\endr
where $r_0\in(0,1)$ sufficiently large such that $\Delta(z,r)\subset\Delta(z_k,r_0)$ for all $z\in\Delta(z_k,\delta)$ and all $k$,
and hence
\begr
\left(\frac{1}{|I_a|}\int_{Q(a)}\left(\frac{\mu(\Delta(z,r))}{(1-|z|^2)^{1+p}}\right)^{\frac{2}{2-p}}\frac{dA(z)}{1-|z|}\right)^{\frac{2-p}{2}}
\lesssim\|D_\varphi-D_\psi\|^q\nonumber
\endr
for any $|a|>R$. For $|a|\leq R$, by (\ref{All})
\begr
&&\frac{1}{|I_a|}\int_{Q(a)}\left(\frac{\mu(\Delta(z,r))}{(1-|z|^2)^{1+p}}\right)^{\frac{2}{2-p}}\frac{dA(z)}{1-|z|}\nonumber\\
&\lesssim&\int_{\D}\left(\frac{\mu(\Delta(z,r))}{(1-|z|^2)^{1+p}}\right)^{\frac{2}{2-p}}\frac{dA(z)}{1-|z|}\nonumber\\
&\leq&\sum_{k=1}^\infty\int_{\Delta(z_k,\delta)}\left(\frac{\mu(\Delta(z,r))}{(1-|z|^2)^{1+p}}\right)^{\frac{2}{2-p}}\frac{dA(z)}{1-|z|}\nonumber\\
&\lesssim&\sum_{k=1}^\infty\left(\frac{\mu(\Delta(z_k,r_0))}{(1-|z_k|^2)^{1+p}}\right)^{\frac{2}{2-p}}(1-|z_k|)\nonumber\\
&\lesssim&\|D_\varphi-D_\psi\|^{\frac{2p}{2-p}}, \nonumber
\endr
and then we get that $\left(\frac{\mu(\Delta(z,r))}{(1-|z|^2)^{1+p}}\right)^{\frac{2}{2-p}}\frac{dA(z)}{1-|z|}$ is a Carleson measure.

On the other hand, denote $$F(w)=|f^\prime(w)|^q(1-|w|^2)^q, ~~~~\Phi_{\mu}(w)=\frac{\mu(\Delta(w,r))}{(1-|w|)^{1+q}}.$$
 Then $F\in T^{1}_{\frac{2}{p}}(h)$ and $\|F\|_{T^{1}_{\frac{2}{p}}(h)}\lesssim\|f\|_{H^p}^p$.
Therefore, by (\ref{T2}) and Lemma \ref{dual} we have
\begr
&&\|(D_\varphi-D_\psi)f\|_{A^q_\alpha}^q\nonumber\\
&\lesssim&\int_{\D}\left(|f^\prime(w)|(1-|w|^2)\right)^q\frac{\mu(\Delta(w,r))}{(1-|w|^2)^{1+q}}(1-|w|)dh(w)\nonumber\\
&=&\left\langle F,\Phi_\mu\right \rangle_{T^2_2(h)}\lesssim\|\Phi_\mu\|_{T_{\frac{2}{2-p}}^{\infty}(h)}\|f\|_{H^p}^p\nonumber\\
&\lesssim&\sup_{a\in\D}\left(\frac{1}{|I(a)|}\int_{Q(a)}\left(\frac{\mu(\Delta(z,r))}{(1-|z|^2)^{1+p}}\right)^{\frac{2}{2-p}}
\frac{dA(z)}{1-|z|}\right)^{\frac{2-p}{2}}\cdot\|f\|_{H^p}^p.\nonumber
\endr
Then the assertion follows.\msk

(iib)
Let us first consider the sufficiency. Let $\{f_n\}$ be in $H^p$ such that $\|f_n\|_{H^p}\leq1$ and $f_n\to0$ uniformly on compact subsets of $\D$ when $n\rightarrow\infty$. A standard argument shows that
$$\lim_{|a|\to1^-}\frac{1}{|I(a)|}\int_{Q(a)}\Phi_\mu(z)^{\frac{2}{2-p}}(1-|z|)dh(z)=0$$
if and only if
\begr
&&\lim_{R\to1^-}\sup_{a\in\D}\frac{1}{|I(a)|}\int_{Q(a)}\Phi_\mu(z)^{\frac{2}{2-p}}(1-|z|)dh_R(z)\nonumber\\
&=& \lim_{R\to1^-}\|\Phi_\mu\|_{T_{\frac{2}{2-p}}^{\infty}(h_R)}=0.
\endr
So for a fixed $\epsilon>0$, there exists $R_0$ such that
 $$\sup_{R\geq R_0}\|\Phi_{\mu}\|_{T_{\frac{2}{2-p}}^{\infty}(h_R)}<\epsilon^q.$$
 Let $k_0$ be such that
 $\sup_{n\geq k_0}\sup_{|z|\leq R_0}|f_n^\prime(z)|\leq\epsilon.$
 Then, bearing in mind (\ref{T2}) and the inequality $\|F_n\|_{T^{1}_{\frac{2}{q}}(h)}\lesssim\|f_n\|_{H^p}^q$, we get
 \begr
&&
\|(D_\varphi-D_\psi)f_n\|_{A^q_\alpha}^q\nonumber\\
&\lesssim&\epsilon^q\int_{R_0\D}(1-|w|)^q\Phi_{\mu}(w)(1-|w|)dh(w)\nonumber\\
&~~&+\int_{\D}F_n(w)\Phi_{\mu}(w)(1-|w|)dh_{R_0}(w)\nonumber\\
&\lesssim&\epsilon^q\left\langle(1-|w|)^q,\Phi_{\mu}\right \rangle_{T^2_2(h)}+\left\langle F_n, \Phi_{\mu}\right \rangle_{T^2_2(h_{R_0})}\nonumber\\
&\lesssim&\epsilon^q\|(1-|w|)^q\|_{T^{1}_{\frac{2}{q}}(h)}\|\Phi_{\mu}\|_{T_{\frac{2}{2-q}}^{\infty}(h)}+\|F_n\|_{T^{1}_{\frac{2}{q}}(h)} \|\Phi_{\mu}\|_{T_{\frac{2}{2-q}}^{\infty}(h_{R_0})}\nonumber\\
&\lesssim&\epsilon^q.\nonumber
\endr
Therefore, $D_\varphi-D_\psi:H^p\to A^q_\alpha$ is compact.\msk

Next, we prove the necessity.
Since $D_\varphi-D_\psi:H^p \to A^q_\alpha$ is compact, the set $(D_\varphi-D_\psi)\circ S_\lambda(\mathcal{C}T^p_2\{z_k\})$ is relatively compact in $A^q_\alpha$.
Consider $f\in T^p_2\{z_k\}$ with $\|f\|_{T^p_2\{z_k\}}=1$.
We can choose $0<R_0<1$ such that
\begr
\int_{\D\backslash R_0\D}\left|(D_\varphi-D_\psi)S_{\lambda}(f)(z)\right|^qdA_\alpha(z)\lesssim\epsilon^q.\nonumber
\endr
Let $g_t\in T^p_2(\{z_k\})$, such that $g_t(z_k)=f(z_k)r_k(t)$, where $r_k$ denotes the $k$th Rademacher function,  replace now $f(z_k)$ by $g_t(z_k)$ in the above inequality, and integrate with respect to $t$ we obtain
\begr
\int_0^1\int_{\D\backslash R_0\D}\left|(D_\varphi-D_\psi)S_{\lambda}(g_t)(z)\right|^qdA_\alpha(z)dt\lesssim\epsilon^q.\nonumber
\endr
Using Fubini's theorem and   Khinchine's inequality we get
\begin{eqnarray*}
I
=\int_{\D\backslash R_0\D}\bigg(\sum_{k=1}^\infty |f(z_k)|^2\big|\frac{(1-|z_k|)^{\lambda}}{(1-\overline{z_k}\varphi(z))^{\lambda+1}} -\frac{(1-|z_k|)^{\lambda}}{(1-\overline{z_k}\psi(z))^{\lambda+1}}\big|^2 \bigg)^{\frac{q}{2}}dA_\alpha(z)
\lesssim\epsilon^q.
\end{eqnarray*}
Replace now $S_\lambda(f)$ by $S_{\lambda,N}(f)$, we have
\begin{eqnarray*} II
=\int_{\D\backslash R_0\D}\bigg(\sum_{k=1}^\infty |f(z_k)|^2\big|\frac{(1-|z_k|)^{\lambda}}{(1-t_{k,N}\overline{z_k}\varphi(z))^{\lambda+1}} -\frac{(1-|z_k|)^{\lambda}}{(1-t_{k,N}\overline{z_k}\psi(z))^{\lambda+1}}\big|^2 \bigg)^{\frac{q}{2}}dA_\alpha(z)
\lesssim\epsilon^q.
\end{eqnarray*}
Now, for any fixed $r_0\in(0,1)$, by Lemma \ref{KE},
\begin{eqnarray*}
&&I+II \nonumber\\
&\gtrsim &\int_{\D\backslash R_0\D}\Bigg(\sum_{k=1}^\infty |f(z_k)|^2\bigg(\big|\frac{(1-|z_k|)^{\lambda}}{(1-\overline{z_k}\varphi(z))^{\lambda+1}}-\frac{(1-|z_k|)^{\lambda}}{(1-\overline{z_k}\psi(z))^{\lambda+1}}\big|^2\nonumber\\
& &~~~~~\mbox{~~~}~~~ ~~~+\big|\frac{(1-|z_k|)^{\lambda}}{(1-t_{k,N}\overline{z_k}\varphi(z))^{\lambda+1}}-\frac{(1-|z_k|)^{\lambda}}{(1-t_{k,N}\overline{z_k}\psi(z))^{\lambda+1}}\big|^2\bigg) \Bigg)^{\frac{q}{2}}dA_\alpha(z)\nonumber\\
&\gtrsim &\int_{\D\backslash R_0\D}\big(\sum_{k=1}^\infty |f(z_k)|^2\frac{\rho(z)^2\big(\chi_{\varphi^{-1}(\triangle(z_k,r_0))}(z)+\chi_{\psi^{-1}(\triangle(z_k,r_0))}(z)\big)}{(1-|z_k|)^2}\big)^{\frac{q}{2}}dA_\alpha(z)\nonumber\\
&\gtrsim &\int_{\D\backslash R_0\D}\sum_{k=1}^\infty |f(z_k)|^q\frac{\rho(z)^q\big(\chi_{\varphi^{-1}(\triangle(z_k,r_0))}(z)+\chi_{\psi^{-1}(\triangle(z_k,r_0))}(z)\big)}{(1-|z_k|)^{q}}dA_\alpha(z)\nonumber\\
&= &\sum_{k=1}^\infty |f(z_k)|^q\frac{\mu(\triangle(z_k,r_0)\cap (\D\backslash R_0\D))}{(1-|z_k|)^{q}}.\nonumber
\end{eqnarray*}
Therefore
\begr
&&\sum_{k=1}^\infty |f(z_k)|^q\frac{\mu(\triangle(z_k,r_0)\cap (\D\backslash R_0\D))}{(1-|z_k|)^{q}}\lesssim\epsilon^q. \label{IAIIe}
\endr
Since $0<q=p<2$, then $s=\frac{p}{q}=1$ and $t=\frac{2}{p}>1$,   by Lemma \ref{dual}, $T^1_t(\{z_k\})\simeq T^\infty_{t^\prime}(\{z_k\})$ with equivalence of norms. Hence
 $$\sup_{\zeta\in\partial\D}\sup_{a\in\Gamma(\zeta)}\left(\frac{1}{|I(a)|}\sum_{z_k\in Q(a)}\left(\frac{\mu(\Delta(z_k,r_0)\cap \D\backslash R_0\D)}{(1-|z_k|)^{1+p}}\right)^{\frac{2}{2-p}}(1-|z_k|)\right)^{\frac{2-p}{2}}\lesssim\epsilon^q.
$$
Using the fact that the point $z\in\D$ for which $Q(a)\cap\Delta(z,r)\neq\emptyset$ is contained in some $Q(a^\prime)$, where $\arg a=\arg a^\prime$ and $1-|a^\prime|\asymp1-|a|$, for all $|a|>R$, where $R=R(r)\in(0,1)$, yields
\begr
&&\int_{Q(a)}\left(\frac{\mu(\Delta(z,r)\cap (\D\backslash R_0\D))}{(1-|z|^2)^{1+q}}\right)^{\frac{2}{2-q}}\frac{dA(z)}{1-|z|}\nonumber\\
&\leq&\sum_{k:\Delta(z_k,\delta)\cap Q(a)\neq\emptyset}\int_{\Delta(z_k,\delta)}\left(\frac{\mu(\Delta(z,r)\cap (\D\backslash R_0\D))}{(1-|z|^2)^{1+q}}\right)^{\frac{2}{2-q}}\frac{dA(z)}{1-|z|}\nonumber\\
&\lesssim&\sum_{z_k\in Q(a^\prime)}\left(\frac{1}{(1-|z_k|^2)^{1+q}}\right)^{\frac{2}{2-q}}\int_{\Delta(z_k,\delta)}\left(\mu(\Delta(z,r)\cap (\D\backslash R_0\D))\right)^{\frac{2}{2-q}}\frac{dA(z)}{1-|z|}\nonumber\\
&\lesssim&\sum_{z_k\in Q(a^\prime)}\left(\frac{\mu(\Delta(z_k,r_0)\cap (\D\backslash R_0\D))}{(1-|z_k|^2)^{1+q}}\right)^{\frac{2}{2-q}}(1-|z_k|),\nonumber
\endr
where $r_0\in(0,1)$ sufficiently large such that $\Delta(z,r)\subset\Delta(z_k,r_0)$ for all $z\in\Delta(z_k,\delta)$ and all $k$. Hence
\begr
\sup_{|a|>R}\frac{1}{|I(a)|}\int_{Q(a)}\left(\frac{\mu(\Delta(z,r)\cap (\D\backslash R_0\D))}{(1-|z|^2)^{1+p}}\right)^{\frac{2}{2-p}}\frac{dA(z)}{1-|z|}
\lesssim\epsilon^q.\nonumber
\endr
It is easy to see that $\Delta(z,r)\cap \D\backslash R_0\D=\Delta(z,r),$
for any $z\in Q(a)$, when $a$ close enough to $1$. Therefore,
\begr
\lim_{|a|\to1^-}\frac{1}{|I(a)|}\int_{Q(a)}\left(\frac{\mu(\Delta(z,r))}{(1-|z|^2)^{1+p}}\right)^{\frac{2}{2-p}}\frac{dA(z)}{1-|z|}
=0,\nonumber
\endr
which completes the proof of the necessity.\msk

  (iii). We only give a detail proof for (iiib) since the  proof of  (iiia) can be proved by standard modifications of the proof of   (iiib) and hence we omitted it.\msk

 (iiib)   Suppose that $D_\varphi-D_\psi:H^p\to A^q_\alpha$ is compact. Since $2\leq q<p$, we have $s=\frac{p}{q}>1$ and $t=\frac{2}{p}<1$. So by Lemma \ref{dual2}, $(T^s_t(\{z_k\}))^*\simeq T^{s^\prime}_{\infty}(\{z_k\})$ with equivalence of norms. Therefore, (\ref{IAIIe}) yields
 \begr
\int_{\partial\D}\left(\sup_{z\in\Gamma(\zeta)}\left(\frac{\mu(\Delta(z,r)\cap (\D\backslash R_0\D))}{(1-|z|^2)^{1+q}}\right)^{\frac{p}{p-q}}\right)dm(\zeta)\lesssim\epsilon^q.\nonumber
\endr
There is a $0<R<1$ such that $\Delta(z,r)\cap \D\backslash R_0\D=\Delta(z,r),$
for any $|z|\geq R$. Thus
 $$\lim_{R\to1}\int_{\partial\D}\left(\sup_{z\in\Gamma(\zeta)\cap \{|z|\geq R\}}\left(\frac{\mu(\Delta(z,r))}{(1-|z|^2)^{1+q}}\right)^{\frac{p}{p-q}}\right)dm(\zeta)=0.$$

  Now, we prove the necessity. Let $\{f_n\}$ be in $H^p$ such that $\|f_n\|_{H^p}\leq1$ and $f_n\to0$ uniformly on compact subsets of $\D$. We can assume that $f_n(0)=0$.
 It is enough to show that
\begr\lim_{n\to\infty}\|(D_\varphi-D_\psi)f_n\|_{A^q_\alpha}=0.\label{C}\endr
  Fix $\epsilon>0$. There exists $R_0$ such that
 \begr\label{min}
\left(\int_{\partial\D}\left(\sup_{z\in\Gamma(\zeta)\cap \{|z|\geq R_0\}}\left(\frac{\mu(\Delta(z,r))}{(1-|z|^2)^{1+q}}\right)^{\frac{p}{p-q}}\right)dm(\zeta)\right)^{\frac{p-q}{p}}
<\epsilon^q.
\endr
 Next, choose $k_0$ such that
 \begr
 \sup_{n\geq k_0}\sup_{|z|\leq R_0}|f_n^\prime(z)|\leq\epsilon.\label{D}
 \endr
 Denote $F_n(w)=|f_n^\prime(w)|^q(1-|w|^2)^q$ and $\Phi_{\mu}(w)=\frac{\mu(\Delta(w,r))}{(1-|w|)^{1+q}}$.
Then, bearing in mind (\ref{T2}), we get   
 \begr
\|(D_\varphi-D_\psi)f_n\|_{A^q_\alpha}^q&\lesssim&\int_{\D}F_n(w)\Phi_{\mu}(w)(1-|w|)dh(w)\nonumber\\
&\asymp&\int_{\partial\D}\left(\int_{\Gamma(\zeta)}F_n(w)\Phi_{\mu}(w)dh(w)\right)dm(\zeta)\nonumber\\
&=&\int_{\partial\D}\left(\int_{\Gamma(\zeta)\cap\{|w|\leq R_0\}}F_n(w)\Phi_{\mu}(w)dh(w)\right)dm(\zeta)\nonumber\\
&~~&+\int_{\partial\D}\left(\int_{\Gamma(\zeta)\cap\{|w|>R_0 \}}F_n(w)\Phi_{\mu}(w)dh(w)\right)dm(\zeta)\nonumber\\
&:=&J_1+J_2.\nonumber
\endr

Since $\mu$ is a finite measure, by (\ref{D}),
 \begr
J_1\lesssim\epsilon^q\int_{\partial\D}\left(\int_{\Gamma(\zeta)\cap\{|w|\leq R_0\}}\mu(\Delta(w,r))\frac{dA(w)}{(1-|w|)^3}\right)dm(\zeta)\lesssim\epsilon^q.\nonumber
\endr
Let $\delta_1=\frac{2\delta}{1+\delta^2}$. Using the fact that $\Delta(z_k,\delta)\subset\Delta(w,\delta_1)$ for any $w\in\Delta(z_k,\delta)$ and (\ref{R}), we get
\begin{eqnarray*}
J_2&\leq&
\int_{\partial\D}\sum_{k:\Delta(z_k,\delta)\cap\Gamma(\zeta)\neq\emptyset}\int_{\Delta(z_k,\delta)\cap\{|w|> R_0\}}F_n(w)\Phi_{\mu}(w)dh(w)dm(\zeta)\nonumber\\
&\lesssim&
\int_{\partial\D}\big(\sup_{w\in\Gamma_\gamma(\zeta)\cap\{|w|> R_0\}}\Phi_{\mu}(w)\big)\cdot\big(\sum_{k:\Delta(z_k,\delta)\cap\Gamma(\zeta)\neq\emptyset}\sup_{w\in\Delta(z_k,\delta)}F_n(w)\big)dm(\zeta).\nonumber
\end{eqnarray*}
By Lemma \ref{L}, we get
\begr
\sup_{w\in\Delta(z_k,\delta)}F_n(w)\lesssim\int_{\Delta(z_k,\delta_1)}\Delta |f_n|^q(z)dA(z).\nonumber
\endr
Therefore, using H{\"o}lder's inequality and (\ref{min}), we have
 \begr
J_2
&\lesssim&
\left(\int_{\partial\D}\sup_{w\in\Gamma_\gamma(\zeta)\cap\{|w|> R_0\}}\Phi_{\mu}(w)^{s^\prime}dm(\zeta)\right)^{\frac{1}{s^\prime}}\nonumber\\
&~~~~&\cdot
\left(\int_{\partial\D}\left(\sum_{k:\Delta(z_k,\delta)\cap\Gamma(\zeta)\neq\emptyset}\int_{\Delta(z_k,\delta_1)}\Delta |f_n|^q(z)dA(z)\right)^{s}dm(\zeta)\right)^{\frac{1}{s}}\nonumber\\
&\lesssim&
\epsilon^q\left(\int_{\partial\D}\left(\sum_{k:\Delta(z_k,\delta)\cap\Gamma(\zeta)\neq\emptyset}\int_{\Delta(z_k,\delta_1)}\Delta |f_n|^q(z)dA(z)\right)^{s}dm(\zeta)\right)^{\frac{1}{s}}. \nonumber
\endr
Let $z\in\Delta(z_k,\delta)\cap\Gamma(\zeta)$. Then there exists $0<r_1<1$ depending only on $\delta$ such that
$\Delta(z_k,\delta_1)\subset\Delta(z,r_1)$ and
$$\bigcup_{k:\Delta(z_k,\delta)\cap\Gamma(\zeta)\neq\emptyset}\Delta(z_k,\delta_1)\subset\bigcup_{z\in\Gamma(\zeta)}\Delta(z,r_1)\subset\Gamma_\gamma(\zeta).$$
Meanwhile, there exists a positive integer $M$, such that for each $z\in\Gamma_\gamma(\zeta)$  belongs to at most $M$  pseudo-hyperbolic disks $\Delta(z_k,\delta_1)$.
Therefore, by Lemma \ref{eq},
 \begr
J_2
&\lesssim&
\epsilon^q\left(\int_{\partial\D}\left(\int_{\Gamma_\gamma(\zeta)}\Delta |f_n|^q(z)dA(z)\right)^{s}dm(\zeta)\right)^{\frac{1}{s}}\nonumber\\
&\lesssim&\epsilon^q\|f_n\|_{H^p}^q\lesssim\epsilon^q.\nonumber
\endr
\msk
Letting $\epsilon\to0$, we get (\ref{C}). The proof is complete.
\end{proof}

By combining Theorems \ref{Th2} and   \ref{Th3}, we give a complete proof for Theorem 1,  our main result in this paper.\\

\textbf{Acknowledgements}: This  project was partially supported by the NNSF of China (No. 11901271) and a grant of Lingnan Normal University (No. 1170919634).\msk

\end{document}